\documentclass[11pt]{article}
\title{Radius estimates for nearly stable $H$-hypersurfaces of dimension 2, 3, and 4.}
\author{G.\ Tinaglia and A.\ Zhou}
\date{}

\usepackage[utf8]{inputenc}
\usepackage{graphicx}
\usepackage{booktabs}
\usepackage{tabularx}
\usepackage{pgfplots}
\usepackage{setspace}
\pgfplotsset{width=8cm, compat=1.9}
\graphicspath{ {images/} }

\usepackage[a4paper,top=25mm,bottom=25mm,left=32mm,right=25mm]{geometry}

\usepackage{mathtools}
\usepackage{amsthm}
\usepackage{amsfonts}
\usepackage{amssymb}
\usepackage{lipsum}
\usepackage{hyperref}

\newtheorem{theorem}{Theorem}[section]

\newtheorem{corollary}[theorem]{Corollary}
\newtheorem{remark}[theorem]{Remark}

\theoremstyle{definition}
\newtheorem{definition}{Definition}[section]

\newcommand{\R}{{\mathbb{R}}}

\DeclareMathOperator{\dist}{dist}
\DeclareMathOperator{\inj}{inj}

\begin{document}

\maketitle

\begin{abstract}
In this paper we study the geometry of complete constant mean curvature (CMC) hypersurfaces immersed in an $(n+1)$-dimensional Riemannian manifold $N$ ($n=2, 3$ and $4$) with sectional curvatures uniformly bounded from below. We generalise radius estimates given by Rosenberg~\cite{ros} ($n=2$) and by Elbert, Nelli and Rosenberg~\cite{nelliestimate}  and Cheng~\cite{cheng} ($n=3,4$) to nearly stable CMC hypersurfaces immersed in $N$. We also prove that certain CMC hypersurfaces effectively embedded in $N$ must be proper.
\end{abstract}

\section{Introduction}
Throughout this paper, we refer to a hypersurface $M$ immersed in a manifold $N$ with constant mean curvature $H$ as an $H$-hypersurface.
Let $N$ be an $(n+1)$-dimensional Riemannian manifold $N$ ($n=2, 3$ and $4$) with sectional curvatures uniformly bounded from below. In their seminal papers, Rosenberg~\cite{ros} ($n=2$, see also~\cite{nelros,rosros}) and Elbert, Nelli and Rosenberg~\cite{nelliestimate}  and Cheng~\cite{cheng} ($n=3,4$) prove radius estimates for stable $H$-hypersurfaces immersed in $N$. In this paper we generalise these estimates to nearly stable $H$-hypersurfaces.

\begin{theorem}\label{radiusest}
    Let \(N\) be an $(n+1)$-dimensional Riemannian manifold ($n=2, 3$ and $4$) with sectional curvatures uniformly bounded from below and let \(M\) be a complete, \(\delta_n\)-stable, $H$-hypersurface with  \(\delta_n <\frac{27}{32}, \frac{7}{12}, \frac{19}{64}\) respectively. Then, if \(|H| > 2\sqrt{|\min(0,\mathcal K)|}\) (where $\mathcal K:=\mathcal{K}(N)$ denotes  the infimum of the sectional
curvatures of $N$), there exists a constant \(c := c(n,\delta_n, H,\mathcal K) > 0\) such that for any \(p \in M\),
    \[ \dist_M(p,\partial M) \leq c. \]
\end{theorem}

See Section~\ref{generalize} for a result involving the scalar curvature of $N$ when $n=2$ that generalizes the main theorem in~\cite{ros}. 

Near stability was a notion widely employed in Colding-Minicozzi Theory, that is  \cite{colding2002space1,colding2002space2,colding2002space3,colding2002space4, colding2012space}, to study the geometry of embedded minimal ($H=0$) disks. Many results about minimal hypersurfaces have employed near stability directly or extended the concept of stability to near stability (see for instance~\cite{cmpr, chezho,cl21,fu09,holiwa,mpr06,tz09}). 

In Section~\ref{proof} we use the radius estimates mentioned above for stable $H$-surfaces together with the Stable Limit Leaf Theorem by Meeks, Perez and Ros~\cite{mpr18} to prove the following theorem. 

\begin{theorem}\label{main}
    With $N$ as in Theorem~\ref{radiusest}, let \(M\) be a complete $H$-hypersurface effectively embedded in \(N\). Suppose that the norm of the second fundamental form of \(M\) is locally bounded (bounded in compact extrinsic balls) and  \(|H| > 2\sqrt{|\min(0,\mathcal{K})|}\). Then \(M\) is proper. 
\end{theorem} 

See Remark~\ref{remark} for a stronger statement when $n=2$    and~\cite{cmt1, cmt2, rt} for examples of complete $H$-surfaces embedded in $\mathbb H^3$ and $\mathbb H^2\times \mathbb R$ that are not proper.

Theorem~\ref{main} is motivated by several results in the literature. In their seminal paper~\cite{coldingcalabi}, Colding and Minicozzi proved that a complete, minimal surface embedded in \(\R^3\) with finite topology must be proper, see also~\cite{colding2002space1,colding2002space2,colding2002space3,colding2002space4,colding2012space}. Meeks and Rosenberg generalised this to complete minimal surfaces embedded in \(\R^3\) with positive injectivity radius~\cite{meekslamination}, see also~\cite{heli}. Finally, Meeks and Tinaglia further generalised both these results to constant mean curvature (CMC) surfaces~\cite{meekscurvature}, see also~\cite{Meeks_2017,metisurfaces,meeks2015onesided,meeks2016limit,metijems}.

\section{Radius estimates for nearly stable \texorpdfstring{$H$}{H}-hypersurfaces}\label{generalize}

We begin this section by reminding the reader of the notion of $\delta$-stability.

\begin{definition}
\label{def:almost-stability}
For $\delta \in [0, 1]$, we say that a $H$-hypersurface $M$ immersed in $N$ is $\delta$-stable if
\[
\int_M \left( |\nabla f|^2 - (1 - \delta)(|A|^2 + \overline{Ric}(\nu)) f^2 \right) \geq 0,
\]
for $f \in C_0^\infty(M)$.

When $\delta=0$, then $M$ is stable.
\end{definition}

In what follows, we generalize the radius estimates given in Theorem 1 in~\cite{ros} and Theorem 1 in \cite{nelliestimate} for stable $H$-hypersurfaces to $\delta$-stable $H$-hypersurfaces. We begin by generalizing Theorem 1 in \cite{nelliestimate}, that is the Theorem~\ref{radiusest} when $\delta=0$.

\begin{proof}[Proof of Theorem~\ref{radiusest}.]
    Our proof draws from the methods established in~\cite{nelliestimate}.
    
    Since \(M\) is \(\delta\)-stable, we can find a smooth function \(u > 0\) on \(M\) such that the $\delta$-stability operator satisfies
    \[ L^{\delta} u = \Delta u + (1-\delta)(|A|^2 + \overline{Ric}(\nu))u = 0, \]
    see for instance Lemma 2.1 in~\cite{wmjpar1}. 
    By decomposing the symmetric shape operator into the mean curvature and the trace-less part, \(A = HI + \Phi\), the square norm of \(\Phi\) is
    \[ |\Phi|^2  = |A|^2 - n|H|^2, \]
    so we can write the near-stability operator as 
    \[ L^{\delta} = \Delta + (1-\delta)(|\Phi^2| + nH^2 + \overline{Ric}(\nu)). \]
    We use \(ds^2\) to denote the induced metric on \(M\) by \(N\) and conformally change the metric to \(d\widetilde{s}^2 = u^{2k}ds^2\), where we will choose \(k\) later. Fix \(p \in M\) and take \(r > 0\) small enough such that the geodesic ball \(B_M(p,r)\) centred at \(p\) and of \(ds\)-radius \(r\) is contained in the interior of \(M\). Let \(\gamma\) be a \(d\widetilde{s}\)-minimising geodesic which joins \(p\) to \(\partial B_M(p,r)\). Let \(a\) be the \(ds\)-length of \(\gamma\) and \(\widetilde{a}\) be the \(d\widetilde{s}\)-length of \(\gamma\). Then we have \(a \geq r\) and it suffices to prove that there exists a constant \(c = c(n,H,K,\delta) > 0\) such that \(a \leq c\).
    To this end, let \(R\) and \(\widetilde{R}\) be the curvature tensors of \(M\) in the metrics \(ds\) and \(d\widetilde{s}\) respectively. Choose an orthonormal basis \(\{\widetilde{e}_1 = \frac{\partial \gamma}{\partial \widetilde{s}}, \widetilde{e}_2,\dots.\widetilde{e}_n\}\) for \(d\widetilde{s}\) such that \(\widetilde{e}_2,\dots.\widetilde{e}_n\) are parallel along \(\gamma\) and let \(\widetilde{e}_{n+1} = \nu\). This yields an orthonormal basis \(\{e_1 = \frac{\partial \gamma}{\partial s} = u^k\widetilde{e}_1, e_2 = u^k\widetilde{e}_2,\dots, e_n = u^k\widetilde{e}_n \}\) for \(ds\). Let \(\overline{R}\) be the curvature tensor for the ambient manifold \(N\). Using this notation, $R_{11}$ (respectively \(\widetilde{R}_{11}\)) is the Ricci curvature in the direction of $e_1$ for the metric $ds$ (respectively $d\widetilde{s}$), and \(\overline{R}_{n+1,n+1}\) is $\overline{Ric}(\nu)$.

    Since \(\gamma\) is \(d\widetilde{s}\)-minimising, by the second variation formula for length, we have
    \[ \int_0^{\widetilde{a}} \left( (n-1)\left(\frac{d\phi}{d\widetilde{s}}\right) - \widetilde{R}_{11}\phi^2 \right) d\widetilde{s} \geq 0, \]
    for any smooth function \(\phi\) with \(\phi(0) = \phi(\widetilde{a}) = 0\). We use the formula for Ricci curvature under conformal change of metric, (see for instance the appendix in~\cite{nelliestimate} for a full calculation)
    \[ \widetilde{R}_{11} = u^{-2k}\left( R_{11} - k(n-2)(\log u)_{ss} - k\frac{\Delta u}{u} + k\frac{|\nabla u|^2}{u^2} \right). \]
    Then by $\delta$-stability, \(L^{\delta}u = \Delta u + (1-\delta)(|\Phi^2| + nH^2 + \overline{Ric}(\nu))u = 0\), so we can replace the Laplacian term yielding
    \begin{align*} 
        \widetilde{R}_{11} &= u^{-2k}(R_{11} - k(n-2)(\log u)_{ss} \\ 
        &+ k(1-\delta)(|\Phi|^2 + nH^2 + \overline{R}_{n+1,n+1}) + k\frac{|\nabla u|^2}{u^2}). 
    \end{align*}
    Next, the Gauss equation relates the ambient curvature to the intrinsic curvature
    \[ R_{ijij} = \overline{R}_{ijij} + h_{ii}h_{jj} - h_{ij}^2 = \overline{R}_{ijij} + (\Phi_{ii} + H)(\Phi_{jj} + H) - (\Phi_{ij} + H\delta_{ij})^2. \]
    Letting \(i=1\) and summing over \(j=2,\dots,n\) gives
    \begin{align*} 
        &R_{11}  \\
        &= \sum_{j=2}^n \overline{R}_{1j1j} + \sum_{j=2}^n \Phi_{11}\Phi_{jj} + (n-1)\Phi_{11}H + \sum_{j=2}^n \Phi_{jj}H + (n-1)H^2 - \sum_{j=2}^n \Phi_{1j}^2 \\
        &= \sum_{j=2}^n \overline{R}_{1j1j} + \sum_{j=2}^n \Phi_{11}\Phi_{jj} + (n-2)\Phi_{11}H + \sum_{j=1}^n \Phi_{jj}H + (n-1)H^2 - \sum_{j=2}^n \Phi_{1j}^2. 
    \end{align*}
    Using the traceless property \(\sum_{j=1}^n \Phi_{jj} = 0\) on the second and fourth terms gives
    \[ R_{11} = \sum_{j=2}^n \overline{R}_{1j1j} - \Phi_{11}^2 + (n-2)\Phi_{11}H + (n-1)H^2 - \sum_{j=2}^n \Phi_{1j}^2. \] 
    We substitute this expression into the formula for \(\widetilde{R}_{11}\) to obtain
    \begin{align*}
        &\widetilde{R}_{11} \\
        &= u^{-2k}\left(\sum_{j=2}^n \overline{R}_{1j1j} - \Phi_{11}^2 + (n-2)\Phi_{11}H + (n-1)H^2 - \sum_{j=2}^n \Phi_{1j}^2 \right. \\ 
        &- \left.k(n-2)(\log u)_{ss} + k(1-\delta)(|\Phi|^2 + nH^2 + \overline{R}_{n+1,n+1}) + k\frac{|\nabla u|^2}{u^2}\right). \\
        &= u^{-2k}\left(\sum_{j=2}^n \overline{R}_{1j1j} + k(1-\delta)\overline{R}_{n+1,n+1} + (kn(1-\delta) + n-1)H^2 + (n-2)\Phi_{11}H \right) \\
        &+ u^{-2k} \left( k(1-\delta)|\Phi|^2 - \Phi_{11}^2 - \sum_{j=2}^n \Phi_{1j}^2 - k(n-2)(\log u)_{ss} + k\frac{|\nabla u|^2}{u^2} \right).
    \end{align*}
    Now let \(\varphi = \phi \circ \widetilde{s}\) so that \(\varphi(0) = \varphi(a) = 0\). We combine the above expression with the first inequality and \(d\widetilde{s}^2 = u^{2k}ds^2\) to obtain
    \begin{align*}
        &(n-1)\int_0^a (\varphi_s)^2 u^{-k}\,ds \geq \int_0^a \varphi^2 u^{-k} \left( \sum_{j=2}^n \overline{R}_{1j1j} + k(1-\delta)\overline{R}_{n+1,n+1} \right)\,ds \\
        &+ \int_0^a \varphi^2 u^{-k} \\
        &\left( (kn(1-\delta) + n-1)H^2 + (n-2)\Phi_{11}H + k(1-\delta)|\Phi|^2 - \Phi_{11}^2 - \sum_{j=2}^n \Phi_{1j}^2 \right)\,ds \\
        &- \int_0^a \varphi^2 u^{-k} \left( k(n-2)(\log u)_{ss} + k\frac{|\nabla u|^2}{u^2} \right)\,ds.
    \end{align*}
    We replace \(\varphi\) by \(\varphi u^{k/2}\) to eliminate the \(u^{-k}\). Then differentiation gives \((\varphi u^{k/2})_s = \varphi_s u^{k/2} + \frac{k}{2}\varphi u^{(k-2)/2}u_s \) which yields
    \begin{align*}
        &(n-1)\int_0^a (\varphi_s)^2\,ds + k(n-1)\int_0^a \varphi \varphi_s u_s u^{-1} \,ds + \frac{k^2(n-1)}{4}\int_0^a \varphi^2 u_s^2 u^{-2}\,ds \\
        &\geq \int_0^a \varphi^2 \left( \sum_{j=2}^n \overline{R}_{1j1j} + k(1-\delta)\overline{R}_{n+1,n+1} \right)\,ds \\
        &+ \int_0^a \varphi^2 \\
        &\left( (kn(1-\delta) + n-1)H^2 + (n-2)\Phi_{11}H + k(1-\delta)|\Phi|^2 - \Phi_{11}^2 - \sum_{j=2}^n \Phi_{1j}^2 \right)\,ds \\
        &- \int_0^a \varphi^2 \left( k(n-2)(\log u)_{ss} + k\frac{|\nabla u|^2}{u^2} \right)\,ds.
    \end{align*}
    Using the divergence theorem, we have \(\int \varphi^2 (\log u)_{ss}\,ds = -2\int \varphi \varphi_s u_s u^{-1}\,ds\). Furthermore, we have \( k^2\int \varphi^2 u_s^2 u^{-2}\,ds = \int \varphi^2 (\log u^k)_s^2\,ds = k^2\int \varphi^2 |\nabla u|^2 u^{-2}\,ds \).
    This allows us to combine the terms in the first and last lines as follows
    \begin{align*}
        &(n-1)\int_0^a (\varphi_s)^2\,ds \geq k(n-3)\int_0^a \varphi \varphi_s u_s u^{-1} \,ds + \left(\frac{1}{k} - \frac{n-1}{4}\right)\int_0^a \varphi^2 (\log u^k)_s^2\,ds \\
        &+ \int_0^a \varphi^2 \left( \sum_{j=2}^n \overline{R}_{1j1j} + k(1-\delta)\overline{R}_{n+1,n+1} \right)\,ds \\
        &+ \int_0^a \varphi^2 \\
        &\left( (kn(1-\delta) + n-1)H^2 + (n-2)\Phi_{11}H + k(1-\delta)|\Phi|^2 - \Phi_{11}^2 - \sum_{j=2}^n \Phi_{1j}^2 \right)\,ds.
    \end{align*}
    We now use the basic inequality \(a^2 + b^2 \geq -ab\) with \(a = (n-2)H\) and \(b = \Phi_{11}/2\) which yields
    \[ (n-2)^2H^2 + \frac{\Phi_{11}^2}{4} \geq -(n-2)H\Phi_{11}. \]
    Replacing this in our inequality gives
    \begin{align*}
        &(n-1)\int_0^a (\varphi_s)^2\,ds \geq k(n-3)\int_0^a \varphi \varphi_s u_s u^{-1} \,ds + \left(\frac{1}{k} - \frac{n-1}{4}\right)\int_0^a \varphi^2 (\log u^k)_s^2\,ds \\
        &+ \int_0^a \varphi^2 \left( \sum_{j=2}^n \overline{R}_{1j1j} + k(1-\delta)\overline{R}_{n+1,n+1}
        + (kn(1-\delta) - n^2 + 5n - 5)H^2 \right)\,ds \\
        &+ \int_0^a \left( k(1-\delta)|\Phi|^2 - \frac{5}{4}\Phi_{11}^2 - \sum_{j=2}^n \Phi_{1j}^2 \right)\,ds.
    \end{align*}
    We claim that the last term is greater than zero. Using the crude estimate
    \[ |\Phi|^2 \geq \sum_{j=1}^n \Phi_{jj}^2 + 2\sum_{j=2}^n \Phi_{1j}^2, \]
    and the traceless property \(\sum_{j=1}^n \Phi_{jj} = 0\) gives us
    \[ |\Phi|^2 \geq \frac{n}{n-1}\Phi_{11}^2 + 2\sum_{j=2}^n \Phi_{1j}^2. \]
    We now need to choose 
    \begin{equation}\label{k1}
    k > \frac{5(n-1)}{4n(1-\delta)}
    \end{equation}
    and combine it with the last inequality to estimate the last term as
    \begin{align*}
        &k(1-\delta)|\Phi^2|^2 - \frac{5}{4}\Phi_{11}^2 - \sum_{j=2}^n \Phi_{1j}^2 ) \\
        &\geq \frac{5}{4}\Phi_{11}^2 + \frac{5(n-1)}{2n}\sum_{j=2}^n \Phi_{1j}^2 - \frac{5}{4}\Phi_{11}^2 - \sum_{j=2}^n \Phi_{1j}^2 = \frac{3n-5}{2n}\sum_{j=2}^n \Phi^2_{1j} \geq 0,
    \end{align*}
    as required. Consequently, we now have
    \begin{align*}
        &(n-1)\int_0^a (\varphi_s)^2\,ds \geq k(n-3)\int_0^a \varphi \varphi_s u_s u^{-1} \,ds + \left(\frac{1}{k} - \frac{n-1}{4}\right)\int_0^a \varphi^2 (\log u^k)_s^2\,ds \\
        &+ \int_0^a \varphi^2 \left( \sum_{j=2}^n \overline{R}_{1j1j} + k(1-\delta)\overline{R}_{n+1,n+1}
        + (kn(1-\delta) - n^2 + 5n - 5)H^2 \right)\,ds.
    \end{align*}
    After choosing $k$ such that
    \begin{equation}\label{k2}
    \frac{1}{k} - \frac{n-1}{4}> 0 \quad(\text{that is } k<\frac 4{n-1}),
    \end{equation}
    we can use the inequality \(a^2 + b^2 \geq -ab\) again with \(a = (\frac{1}{k} - \frac{n-1}{4})^{1/2} \varphi (\log u^k)_s\) and \(b = \frac{n-3}{2}(\frac{1}{k} - \frac{n-1}{4})^{-1/2} \varphi_s\) to obtain
    \[ \left(\frac{1}{k} - \frac{n-1}{4}\right)\varphi^2 (\log u^k)_s^2 + \frac{(n-3)^2}{4}\left(\frac{1}{k} - \frac{n-1}{4}\right)^{-1}\varphi_s^{2} \geq -(n-3)\varphi \varphi_s (\log u^k)_s. \]
    Hence,
    \begin{align*}
        &(n-1)\int_0^a (\varphi_s)^2\,ds \geq -\frac{(n-3)^2}{4}\left(\frac{1}{k} - \frac{n-1}{4}\right)^{-1} \int_0^a (\varphi_s)^2 \\
        &+ \int_0^a \varphi^2 \left( \sum_{j=2}^n \overline{R}_{1j1j} + k(1-\delta)\overline{R}_{n+1,n+1} + (kn(1-\delta) - n^2 + 5n - 5)H^2 \right)\,ds.
    \end{align*}
   Rearranging the terms, we now have an inequality of the form
    \[ A\int_0^a (\varphi_s)^2\,ds \geq \int_0^a \varphi^2 \left( \sum_{j=2}^n \overline{R}_{1j1j} + k(1-\delta)\overline{R}_{n+1,n+1} + (kn(1-\delta) - n^2 + 5n - 5)H^2 \right)\,ds\]
    where \(A := \frac{4(k(2-n)+(n-1))}{4-k(n-1)}\) is positive, thanks to the condition \eqref{k2}. We now want to choose \(B > 0\) such that 
    \[ B \leq \sum_{j=2}^n \overline{R}_{1j1j} + k(1-\delta)\overline{R}_{n+1,n+1} + (kn(1-\delta) - n^2 + 5n - 5)H^2 \]
    and therefore this would give
   \begin{equation}\label{eq1}
   A\int_0^a (\varphi_s)^2\,ds \geq B\int_0^a \phi^2\,ds.
   \end{equation}
    
    Recall that $\mathcal{K}$ denotes  the infimum of the sectional
curvatures of $N$. If \(\mathcal{K} \geq 0\) we can set \(B := (kn(1-\delta) - n^2 + 5n - 5)H^2\), which is indeed positive if \(|H| > 0\) and \(k > \frac{5(n-1)}{4n(1-\delta)}\) (that is assumption~\eqref{k1}), when $n=2,3$ or $4$. 
    Otherwise, note that
\begin{equation}\label{scalar} \sum_{j=2}^n \overline{R}_{1j1j} + k(1-\delta)\overline{R}_{n+1,n+1} \geq (kn(1-\delta) + n - 1)\mathcal{K}.
\end{equation}
  Therefore  setting $B := (kn(1-\delta) - n^2 + 5n - 5)H^2 + (kn(1-\delta) + n - 1)\mathcal{K}$ we have $$B\leq \sum_{j=2}^n \overline{R}_{1j1j} + k(1-\delta)\overline{R}_{n+1,n+1} + (kn(1-\delta) - n^2 + 5n - 5)H^2$$ and if 
    \[ H^2 > \frac{kn(1-\delta)+n-1}{kn(1-\delta)-n^2+5n-5}\mathcal{K} \] then $B$ is also greater than zero.
    
    Using our assumptions \eqref{k1} and \eqref{k2}, that is \( \frac{5(n-1)}{4n(1-\delta)} < k < \frac{4}{n-1}\), we can estimate the quotient
    \begin{align*} 
        \frac{kn(1-\delta)+n-1}{kn(1-\delta)-n^2+5n-5} &< \frac{16n + 4(n-1)^2}{5(n-1)^2 + 4(n-1)(-n^2 + 5n - 5)} \\
        & = \frac{4(n+1)^2}{(n-1)(n-5)(5-4n)} < 4.
    \end{align*}
    Therefore \(B\) is positive provided that \(|H| > 2\sqrt{|\mathcal K|}\). 
    
    Integrating equation~\eqref{eq1} by parts now gives
    \[ \int_0^a (\varphi_{ss}A + B\varphi)\varphi\,ds \leq 0. \]
    Choose \(\varphi = \sin(\pi sa^{-1})\), for \(s \in [0,a]\) so that
    \[ \int_0^a \left(B - \frac{A\pi^2}{a^2}\right)\sin(\pi sa^{-1})\,ds \leq 0, \]
   which implies 
    \[ B - \frac{A\pi^2}{a^2} \leq 0, \] that is  \(a < \sqrt{A}\pi/\sqrt{B}\). Setting $c:=\sqrt{A}\pi/\sqrt{B}$, this finishes the proof of the theorem, provided that we can show that we can find $\delta>0$ and \(k\) such that
    \begin{equation}\label{k3}
   \frac{5(n-1)}{4n(1-\delta)} < k < \frac{4}{n-1}. 
\end{equation}
   This inequality is consistent for $n=2$ when $\delta<\frac{27}{32}$, for \(n=3\) when \(\delta < \frac{7}{12}\) and for \(n=4\) when \(\delta < \frac{19}{64}\). 
\end{proof}


As in the original paper, we can prove a corollary which asserts the non-existence of certain $H$-hypersurfaces.

\begin{corollary}\label{radestcor}
    With $N$ as in Theorem~\ref{radiusest}, let \(M\) be a complete,   \(\delta_n\)-stable, $H$-hypersurface with \(\delta_n < \frac{27}{32}, \frac{7}{12}, \frac{19}{64}\) respectively.  If \(|H| > 2\sqrt{|\min(0,\mathcal K)|}\), then \(\partial M \neq \emptyset\).
\end{corollary}

\begin{proof}
    From the previous theorem, the radius of an intrinsic geodesic disk of \(M\) that does not meet \(\partial M\) is at most \(c = c > 0\). Assuming that the boundary of \(M\) is empty, the diameter of \(M\) is at most \(c\), so the Hopf-Rinow theorem implies that \(M\) is compact. By \(\delta\)-stability, there exists a function \(u > 0\) on \(M\) such that \(L^{\delta}u = 0\). By compactness, let \(p \in M\) be the minimum of the function \(u\). Then
    \[ 0 \leq \Delta u(p) = -(1-\delta)(|\Phi|^2(p) + nH^2 + \overline{R}_{n+1,n+1}(p))u(p).\]
    The choice of \(H\) guarantees that the potential is positive, hence the right hand side is negative which yields a contradiction.
\end{proof}

Next we generalize Theorem 1 in~\cite{ros}, that is the theorem below when $\delta=0$.

\begin{theorem}\label{radest3}
    Let \(N\) be a $3$-dimensional Riemannian manifold with scalar curvature uniformly bounded from below by $\mathcal S$ and let \(M\) be a complete, \(\delta\)-stable, $H$-surface with \(\delta < \frac34\). If $3H^2+\mathcal S> 0$ then for any \(p \in M\),
    \[ \dist_M(p,\partial M) \leq 2\pi\sqrt{\frac{1-\delta}{(3-4\delta)(3H^2+\mathcal S)}}. \]
\end{theorem}

\begin{proof}
In order to prove this theorem, one can follow the proof of Theorem~\ref{radiusest} with $n=2$. Note that when $n=2$, condition~\eqref{k3} becomes $\frac{5}{8(1-\delta)} < k < 4 $
and thus we can take $k=\frac{1}{1-\delta}$ as long as $\delta<\frac34$. With this choice of $k$, note that instead of equation~\eqref{scalar} we have 
\begin{equation} \sum_{j=2}^n \overline{R}_{1j1j} + k(1-\delta)\overline{R}_{n+1,n+1} = \overline{R}_{1212} + \overline{R}_{3,3}\geq \mathcal S.
\end{equation}
The proof then continues after defining 
$B:=3H^2+\mathcal S$, and noting that when $n=2$, we have $A:=4\frac{1-\delta}{3-4\delta}$.
\end{proof}

Just like before, one can prove a corollary which asserts the non-existence of certain $H$-surfaces.

\begin{corollary}\label{radestcor3}
    With $N$ as in Theorem~\ref{radest3},  let \(M\) be a complete, \(\delta\)-stable, $H$-surface with \(\delta < \frac34\). If $3H^2+\mathcal S> 0$, then \(\partial M \neq \emptyset\).
\end{corollary}

\section{Properness of effectively embedded \texorpdfstring{$H$}{H}-hypersurfaces}\label{proof}

 In  this section we prove Theorem~\ref{main}. We begin by defining ``effectively embedded,'' a term that is sometimes also referred to as ``weakly embedded.''
 
 \begin{definition}\label{effemb}
Let $\phi\colon M \looparrowright N$ be an $H$-hypersurface. We say that $M$ is effectively embedded if at any point \( p \in \phi(M) \), there exists \( \epsilon > 0 \) such that either
\begin{enumerate}
    \item $\phi^{-1}(p)$ consists of a single point $p_1\in M$ and the connected component of \( B_N(p,\epsilon) \cap \phi(M) \) containing $p$ is an embedding of the connected component of \( \phi^{-1}(B_N(p,\epsilon) \cap \phi(M)) \) that contains $p_1$, or
    \item $\phi^{-1}(p)$ consists of two points $p_1$ and $p_2$, $\phi$ restricted to the connected component $\Sigma_i$ of \( \phi^{-1}(B_N(p,\epsilon) \cap \phi(M)) \) that contains $p_i$, $i=1,2$, is an embedding, the connected component of \( B_N(p,\epsilon) \cap \phi(M) \) containing $p$ is equal to $\phi(\Sigma_1\cup\Sigma_2)$, $\phi(\Sigma_1)$ and $\phi(\Sigma_2)$ meet  tangentially at $p$ and their mean curvature vectors point in opposite directions.
\end{enumerate}
\end{definition}

Note that if $M$ is embedded, then it is effectively embedded. This definition is natural as it includes limits of a converging sequence of embedded $H$-hypersurfaces, see for example~\cite{bst}. Abusing the notation, when dealing with effectively embedded hypersurfaces, we will ignore the immersion $\phi$ and when Case 2 of Definition~\ref{effemb} occurs, we might refer to either of the $\phi (\Sigma_i)$ (that is $\Sigma_i$), $i=1,2$, as the connected component of $B_N(p,\epsilon) \cap \phi(M)$ (that is $B_N(p,\epsilon) \cap M $) containing $p$.

The proof of Theorem~\ref{main} is going to use the Stable Limit Leaf Theorem in~\cite{mpr18}. To that end, we need to recall a few definitions. The following definition is sometimes referred to as a weak $H$-lamination.

\begin{definition}\label{hlam} Given $H>0$, a codimension one $H$-lamination $\mathcal{L}$ of $N$ is a collection of immersed (not necessarily injectively) $H$-hypersurfaces $\{L_\alpha\}_{\alpha \in I}$, called the leaves of $\mathcal{L}$, satisfying the following properties:

\begin{enumerate}
    \item $\mathcal{L} = \bigcup_{\alpha \in I} \{L_\alpha\}$ is a closed subset of $N$.
 
    \item Given a leaf $L_\alpha$ of $\mathcal{L}$ and a small disk $\Delta \subset L_\alpha$, there exists an $\epsilon > 0$ such that, if $(q, t)$ denote the normal coordinates for $\exp_q(t\eta_q)$ (here $\exp$ is the exponential map of $N$ and $\eta$ is the unit normal vector field to $L_\alpha$ pointing to the mean convex side of $L_\alpha$), then:
    \begin{enumerate}
        \item The exponential map $\exp: U(\Delta,\epsilon) = \{(q,t) \mid q \in \mathrm{Int}(\Delta), t \in (-\epsilon, \epsilon)\}$ is a submersion.
        \item The inverse image $\exp^{-1}(\mathcal{L}) \cap \{q \in \mathrm{Int}(\Delta), t \in [0, \epsilon)\}$ is a lamination of $U(\Delta,\epsilon)$.
    \end{enumerate}
\end{enumerate}
\end{definition}

\begin{definition}
    Let \(\mathcal{L}\) be an \(H\)-lamination of \(N\) and let \(L\) be a leaf of \(\mathcal{L}\). We say that \(L\) is a \emph{limit leaf} if \(L\) is contained in the closure of \(\mathcal{L} - L\).
\end{definition}

A properly effectively embedded \(H\)-hypersurface is an $H$-lamination with one leaf. We can now state the Stable Limit Leaf Theorem.

\begin{theorem}[Theorem 1 in~\cite{mpr18}]\label{stableleaf}
    The limit leaves of a codimension one H-lamination of a Riemannian manifold are stable.
\end{theorem}

Finally, we are ready to begin the proof of Theorem~\ref{main}. 

\begin{proof}[Proof of Theorem~\ref{main}] Recall that $M$ having locally bounded norm of the second fundamental form means that the intersection of $M$ with any closed extrinsic ball of $N$ has norm of the second fundamental form bounded from above by a constant that only depends on the ball.

Arguing by contradiction, suppose $M$ is not proper. The first step in the proof is to observe that \(\overline{M}\), the closure of $M$, has the structure of an $H$-lamination.  

Let $p\in \overline{M}$. Since $|A|$ is locally bounded, we can apply Theorem~\ref{saturnino} to give a sufficiently small harmonic chart $(U, \phi, B_N(p, r))$ such that for any $\epsilon \in (0, r)$, if $\Sigma$ denotes the set of connected components of $\phi^{-1}(M \cap B_N(p, r))$ intersecting the Euclidean ball $B_{\mathbb{R}^3}(0, \epsilon)$, then there exist $\epsilon > 0$, $\rho \in (\epsilon, r)$, and $C' < \infty$ and a rotation $R \in O(\mathbb{R}^3)$ such that:

\begin{enumerate}
    \item Every connected component of $\Sigma:=R(\Sigma) \cap B(\rho) \times \mathbb{R}$ is the graph of a function $u$ over $B(\rho)$.
    \item For all such functions $u$, we have $\|u\|_{C^{2,\alpha}(B(\rho))} \leq C'$.
\end{enumerate}

Note that $\Sigma=\Sigma_+\cup\Sigma_-$, where $\Sigma_+$ ($\Sigma_-$ respectively) is the collection of components whose mean curvature vector is pointing ``upward'' (``downward'').

If $p\in M$ and $M$ is proper in a neighbourhood $W$ of $p$ 
then $M\cap W=\overline M\cap W$ and, after possibly using a smaller chart, $\Sigma$ consists of a finite number of connected components and $\overline{M}$ has the structure of an $H$-lamination in a neighbourhood of $p$. Else, the number of connected components in $\Sigma$ is infinite and there exists a sequence of connected components $\sigma_n\in \Sigma$ such that $p\notin \sigma_n$ but $p\in\lim_{n\to\infty} \sigma_n$.
After passing to a sub-sequence and without loss of generality, we can assume that $\sigma_n\in \Sigma_+$, for all $n$. Moreover, a standard compactness argument gives that $\sigma_n$ converges $C^{2,\alpha}$ to a graph $\sigma$ containing $p$, with constant mean curvature $H$, bounded $|A|$ and upward pointing mean curvature vector. Note that if $\phi^{-1}(\sigma) \not\subset M$ then, by definition, $\phi^{-1}(\sigma)\subset \overline M$. Furthermore, observe that if $\sigma_-$ is a component in $\Sigma_-$ that is above $\sigma$, then $\sigma_-$ cannot be arbitrarily close to $p$, see for instance~\cite[Lemma 3.1]{meeks2010existence}. Therefore, after possibly using an even smaller chart, no component of $\Sigma_-$ is above $\sigma$. Therefore, by taking $\Delta=\sigma$ in Definition~\ref{hlam}, this discussion shows that $\overline M$ satisfies condition 2 of Definition~\ref{hlam}.

Using more or less the previous arguments, it is fairly standard to show that $\overline M-M$, and therefore $\overline M$ is a collection of $H$-hypersurfaces. Finally, by definition, $\overline M$ is a closed subset of $N$. This finishes to prove the observation that $\mathcal L:=\overline M$ is an $H$-lamination.

 Next, we claim that $\overline M\neq M$. Arguing by contradiction, suppose that \(\overline M=M\) is the only leaf in the lamination $\mathcal L$. Since $M$ is not proper and \(\overline M=M\), $M$ contains a limit point, namely there is a point  $p\in M$ and a sequence of points $p_n\in M$ converging to $p$ extrinsically, but not intrinsically. It's not hard to see that the set of limit points is open and closed. This implies that every point \(x \in M\) is a limit point.
Let $U$ in $M$ be a small geodesic ball centred at \(x\) which is the limit as $n$ goes to infinity of a sequence of pairwise disjoint balls \(U_n\) in \(M\) centered at $p_n$. Performing this construction again for all \(p_n \in U_n\), we obtain pairwise disjoint balls \(U_{n,m} \subset M\) which converge to \(U_n\) in \(N\) as \(m \) goes to infinity. We can then repeat this process, for example next on \(U_{n,m}\) and again on the resulting balls. This iterative process yields an uncountable number of disjoint balls on \(M\) which contradicts the second countable intrinsic property of the topology of a manifold. This shows that $\overline M-M$ is not empty.

Let $L$ be a leaf  $\overline M-M$. Then, by definition, $L$ is a limit leaf of $\mathcal L$. By Theorem~\ref{stableleaf}, $L$ is stable and by Theorem 1 in~\cite{ros}, when $N$ is a 3-dimensional manifold (see Remark~\ref{remark}), and Theorem 1 in \cite{nelliestimate}, when $N$ is a 4 or 5-dimensional manifold,  \(L\) cannot exist (see Corollary~\ref{radestcor} with $\delta=0$). 
This final contradiction proves that $M$ must be proper.
\end{proof}

\begin{remark}\label{remark}
    Analogously to what happened in Section~\ref{generalize}, thanks to Theorem 1 in~\cite{ros} (see Theorem~\ref{radest3} and Corollary~\ref{radestcor3} with $\delta=0$) when the dimension of $N$ is $3$, one can replace the condition  \(|H| > 2\sqrt{|\min(0,\mathcal{K})|}\) with $3H^2+\mathcal S> 0$ (where $\mathcal S$ is a uniform bound from below for the scalar curvature of $N$), and prove a stronger properness result. 
\end{remark}

\section{Properties of  effectively  embedded \texorpdfstring{$H$}{H}-hypersurfaces with bounded \texorpdfstring{$\|A\|$}{|A|}}

To keep the paper self-contained, in this section we state a few properties of hypersurfaces effectively embedded in a manifold with bounded second fundamental form that were used in the proof of Theorem~\ref{main}. These statements are generalizations of existing results.

Given a point $x$ in a hypersurface $M$ in $\mathbb R^n$, a neighbourhood of $x$ is always graphical over the tangent plane of $M$ at $x$. However, the size of such neighbourhood depends on $x$ and, in general, it could be very small. However, when the norm of the second fundamental form of $M$ is bounded, then the size of such neighbourhood is uniformly bounded from below independently of the point, see for example~\cite{ti}.

Analogous results are true for hypersurfaces in a manifold $N$. We begin by referencing a result by Hebey and Herzlich~\cite{ehmh} that establishes that the metric respect to some harmonic coordinates is locally uniformly \(C^{1,\alpha}\)-controlled for any \(\alpha \in (0,1)\), depending only on the bounds on the injectivity radius and sectional curvatures of \(N\). The version stated below is presented by Rosenberg, Souam and Toubiana in the appendix of~\cite{rosenberg2009general}. Note that the version in~\cite{rosenberg2009general} is stated for 3-dimensional manifolds. It is not hard to see that the proof works in higher dimensions~\cite{ehmh}.

\begin{theorem}\label{rosenberg}
    Let \(\alpha \in (0,1)\) and \(\delta > 0\). Let \((N,g)\) be a complete Riemannian manifold with absolute sectional curvature bounds $|K| \leq \Lambda < \infty$. Let \(\Omega\) be an open subset of \(N\) and define the fattening
    \[ \Omega(\delta) := \{x \in N : \dist_{N}(x,\Omega) < \delta \}.  \]
    Suppose that there exists an \(i > 0\) such that for all \(x \in \Omega(\delta)\), we have \(\inj_{N}(x) > i\). Then there exists a constant \(Q_0 > 1\) and a radius \(r_0 > 0\) which depend only on \(i\), \(\delta\), \(\Lambda\) and \(\alpha\), but not on \(N\), such that for any \(x \in \Omega\), there exists a harmonic chart \((U,\phi,B_{N}(x,r_0))\) with \(\phi(0) = x\). Furthermore, we have \(C^{1,\alpha}\)-control over the metric tensor, that is
    \[ Q_0^{-1}\delta_{ij} \leq {g}_{ij} \leq Q_0\delta{ij} \]
    as quadratic forms, and \(\|(\phi^*{g})_{ij}\|_{C^{1,\alpha}(U)} \leq Q_0\)
\end{theorem}

Using the result above, by transferring the problem onto Euclidean space, in the appendix of~\cite{saturnino2021genus}, Saturnino proves the following theorem. Once again this result is stated there for surfaces in a 3-dimensional Riemannian manifold but its proof works in higher dimensions.

\begin{theorem}\label{saturnino} Suppose $(N, g)$ is a  manifold with absolute sectional curvature bounds $|K| \leq \Lambda < \infty$ and let $M \subset N$ be an effectively embedded  $H$-hypersurface. Let $\Omega \subset N$ be an open set lying away from the boundary of $N$, and suppose the norm of the second fundamental form of $M$ in $\Omega$ is bounded above by a constant $C < \infty$. Fix any $\alpha \in (0,1)$ and suppose $\delta$, $i$, $r_0$, and $Q_0$ are as in Theorem~\ref{rosenberg}. Fix an $r \in (0, r_0)$ and let $x \in \Omega$ be such that $d_M(x, \partial M ) > r$. 

Choose a harmonic chart $(U, \phi, B_N(x, r))$ as in Theorem~\ref{rosenberg}. For any $\epsilon \in (0, r)$, let $\Sigma$ be the set of connected components of $\phi^{-1}(M \cap B_N(x, r))$ intersecting the Euclidean ball $B_{\mathbb{R}^3}(0, \epsilon)$. Then there exist $\epsilon > 0$, $\rho \in (\epsilon, r)$, and $C' < \infty$ depending only on $\Lambda$, $C$, $i$, and $\alpha$, and a rotation $R \in O(\mathbb{R}^3)$ such that:

\begin{enumerate}
    \item Every connected component of $R(\Sigma) \cap B(\rho) \times \mathbb{R}$ is the graph of a function $u$ over $B(\rho)$.
    \item For all such functions $u$, we have $\|u\|_{C^{2,\alpha}(B(\rho))} \leq C'$.
\end{enumerate}
\end{theorem}

In fact, in~\cite{saturnino2021genus} the hypersurfaces are assumed to be properly embedded but   the proof works for effectively embedded hypersurfaces. Note that in this more general case the number of connected components in $\Sigma$ could be infinite and connected component must be intended in the sense described in Definition~\ref{effemb}.

\center{Giuseppe Tinaglia at giuseppe.tinaglia@kcl.ac.uk}\\
Department of Mathematics, King's College London,  London, WC2R 2LS, U.K.

\center{Alex Zhou at alex.zhou@kcl.ac.uk}\\
Department of Mathematics, King's College London,  London, WC2R 2LS, U.K.

\bibliographystyle{plain}
\bibliography{ref}

\end{document}